\documentclass[12pt]{amsart}

\usepackage{amsfonts, amssymb, amsmath, latexsym,  mathtools, amscd,mathrsfs}
\usepackage{hyperref,graphicx}
\usepackage{apxproof}
\usepackage{xcolor}
\usepackage{url}
\usepackage{wrapfig}
\usepackage{subcaption,cleveref}
\usepackage{cite}
\usepackage{tikz}

\newtheorem{Theorem}{Theorem}[section]
\newtheorem{Proposition}[Theorem]{Proposition}
\newtheorem{Lemma}[Theorem]{Lemma}
\newtheorem{Corollary}[Theorem]{Corollary}

\theoremstyle{remark}

\newtheorem{Example}[Theorem]{Example}
\newtheorem{Definition}[Theorem]{Definition}

\newcommand{\CC}{{\mathbb C}}

\newcommand{\TT}{{\mathbb T}}
\newcommand{\ZZ}{{\mathbb Z}}

\newcommand{\calU}{{\mathcal U}}
\newcommand{\calE}{{\mathcal E}}

\newcommand{\calV}{{\mathcal V}}

\newcommand{\Lc}{{\mathcal L}}

\title[Generic Irreducibility of Bloch Varieties]{Generic Irreducibility of Bloch Varieties for Periodic Graph Operators}

\author{Matthew Faust}
\address{Matthew Faust, Department of Mathematics, Michigan State University, East Lansing, MI 48824, USA} \email{mfaust@msu.edu}
\urladdr{https://mattfaust.github.io/}
\author{Wencai Liu}
\address{Wencai Liu, Department of Mathematics,
         Texas A\&M University, College Station, Texas 77843,  USA}
\email{wencail@tamu.edu}
\urladdr{https://sites.google.com/view/wencail/home}

\begin{document}

\begin{abstract}
We give a complete  characterization of generic irreducibility for dispersion polynomials and Bloch varieties of periodic graph operators. More precisely,  we prove that  for a generic choice of edge weights and potentials, the dispersion polynomial/Bloch variety of  a nontrivial periodic graph is irreducible if and only if   the quotient graph is connected.

Our proof uses a strong dichotomy for parameterized Laurent polynomials: reducibility either occurs for every parameter   or   fails on a nonempty Zariski-open set. After establishing this dichotomy, we reduce the problem to minimally connected periodic graphs.
\end{abstract}
\maketitle 


\section{Introduction}

The study of dispersion relations and Bloch varieties has a long history in the literature. We do not attempt to give a complete review here, and instead refer   readers to the survey articles \cite{Kuchment2016Overview,Kuchment2023DispersionSurvey,ShipmanSottile2025} and to \cite{FaustLiu2025RareFlatBands} for  the background and recent development.

Classical irreducibility questions for Bloch varieties go back to the work of  B\"attig , Gieseker, Kn\"orrer,  and Trubowitz, where algebraic-geometric compactification methods were used to study continuous and discrete periodic Schr\"odinger operators \cite{GiesekerKnoerrerTrubowitz1993,KnoerrerTrubowitz1990,Battig1992,BattigKnoerrerTrubowitz1991}.  Recently, the second author initiated a new approach combining  tools from analysis and algebraic geometry, and used it to prove irreducibility for fixed-energy level sets of the Bloch variety, i.e. Fermi varieties, with applications to embedded eigenvalues. Later, irreducibility  of Bloch varieties was studied for more general Schr\"odinger operators on lattice graphs \cite{FillmanLiuMatos2022,FaustLopezGarcia2025,FillmanLiuMatos2024}.

Irreducibility of Bloch varieties plays a crucial role  in the study of spectral properties of periodic graph operators. For instance, it is related to quantum ergodicity and the structure of spectral band functions \cite{Liu2024QuantumErgodicity,McKenzieSabri2023,Liu2022Fermi,KS26}. Irreducibility of Bloch varieties   also appears naturally in inverse spectral problems  for discrete periodic Schr\"odinger operators \cite{Liu2024FermiIsospectrality,LiuBorgToAppear,ChuLyuYang2025, CFK26}. A special case of irreducibility phenomena is the absence of flat bands, which has also been studied extensively for periodic graph operators \cite{KorotyaevSaburova2014,SabriYoussef2023,FaustKachkovskiy2025Absence,FaustLiu2025RareFlatBands}.

This paper continues the study initiated in \cite{FaustLiu2025RareFlatBands}, where flat bands for periodic graph operators were investigated using tools from spectral analysis, algebraic geometry and combinatorics. In particular, our earlier work showed that for connected $\ZZ^d$-periodic graphs, flat bands are generically absent. In the present paper, we study a stronger and more delicate algebraic question: the generic irreducibility of the dispersion polynomial, or equivalently, the generic irreducibility of the corresponding Bloch variety. When the quotient graph has only one vertex, the problem is trivial. So throughout the paper,  we assume  that the quotient graph has at least two vertices.

For a $\ZZ^d$-periodic graph operator, Floquet theory reduces the spectral analysis of the infinite-dimensional operator to the study of a finite matrix $L(z)$ (referred to as the Floquet matrix) whose entries are Laurent polynomials in $z$, where $z\in(\CC^*)^d$ is obtained from the quasimomentum variable $k$ by the change of variables
\[
z_j=e^{2\pi i k_j}, z=(z_1,z_2,\dots,z_d)\text{ and } k=(k_1,k_2,\dots, k_d).
\] The corresponding dispersion polynomial is
\[
D(z,\lambda)=\det(L(z)-\lambda I),
\]
and the Bloch variety is the zero set of $D(z,\lambda)$ in $(\CC^*)^d\times\CC$.

We say a $\ZZ^d$-periodic graph is trivial if there exists a fundamental domain $W$ such that $W$ and $\Gamma\backslash W$ are disconnected. Equivalently, after choosing a suitable fundamental domain, the Floquet matrix is independent of $z$. On the other hand, if the quotient graph is disconnected, then after a suitable ordering of the vertices, the Floquet matrix decomposes into blocks, implying that the dispersion polynomial factors. In this paper, we show that these are the only obstructions: connectedness of the quotient graph and nontriviality  are precisely the natural graph structures that lead to generic irreducibility.

We prove the following theorem.
\begin{Theorem}~\label{Thm:main}
    $D(z,\lambda)$ is generically (i.e. for generic choices of edge weights and potentials) irreducible if and only if the quotient graph of $\Gamma$ is connected and $\Gamma$ is nontrivial.
\end{Theorem}

As a corollary, we prove the corresponding statement for Bloch varieties.
\begin{Corollary}\label{maincor}
  The Bloch variety is generically (i.e. for generic choices of edge weights and potentials)  irreducible if and only if the quotient graph of $\Gamma$ is connected and $\Gamma$ is nontrivial.
\end{Corollary}

The proof combines two main novel ideas. Firstly, in Section~\ref{Sec:2} we establish a strong algebraic dichotomy for parameterized Laurent polynomials: reducibility either occurs for every parameter, or else fails on a nonempty Zariski-open subset of the parameter space. Related generic dichotomies (generically true vs generically not true) have appeared in the study of dispersion relations and flat bands \cite{DoKuchmentSottile2020,FaustLiu2025RareFlatBands,FaustKachkovskiy2025Absence}; here we prove a stronger dichotomy (always true vs generically not true). In particular, to prove generic irreducibility, it is enough to find a single parameter value for which the polynomial is irreducible (see Corollary \ref{cor:one_point}).

A further new ingredient is the algebraic geometry framework of the proof. Rather than studying the parameters of potentials and edge weights   in affine varieties directly, we study the corresponding factorizations of the dispersion polynomial in projective varieties. This allows us to exploit the projective nature of the factorization space, in particular the closedness of images of projective morphisms. Our approach  should be easily adapted   to other spectral problems for periodic graph operators, such as flat bands and the irreducibility of Fermi varieties. We believe the new framework provides a robust algebraic-geometric tool in the spectral theory of periodic operators.

Secondly, we reduce the graph-theoretic problem to a special class of minimally connected periodic graphs. 
For this reduced class, the dispersion polynomial has a particular structure. We then rule out all possible factorizations: linear factors are excluded using  ideas from \cite{FaustLiu2025RareFlatBands}, while the remaining factorizations are excluded by degree and specialization arguments.

Several ingredients from \cite{FaustLiu2025RareFlatBands} are used throughout the argument, especially the Floquet-theoretic setup and the algebraic information encoded by the dispersion polynomial. In this sense, the present paper can be viewed as a continuation of that work,  shifting from flat-band phenomena to generic irreducibility.

Finally, we  would like to remark that  in  \cite{FaustLiu2025RareFlatBands}, we  characterize the periodic graphs for which flat bands are generically absent with respect to both potentials and edge weights. Later, in \cite{FaustKachkovskiy2025Absence}, the first author and Kachkovskiy proved the corresponding result for fixed nonzero edge weights and generic potentials by a different method based on perturbation arguments.
This naturally raises the question of whether the irreducibility results in the present paper can also be strengthened from generic potentials and edge weights to fixed nonzero edge weights and generic potentials. At the end of this paper, we explicitly construct examples showing that this is not possible in general.

The rest of the paper is organized as follows. In Section~\ref{Sec:2} we prove the strong dichotomy for a family of Laurent  polynomials. In Section~\ref{Sec:MainRes} we  list the background of  $\ZZ^d$-periodic graphs and $\ZZ^d$-periodic graph operators. In Section~\ref{Sec:reduce}, we reduce Theorem~\ref{Thm:main} to a more specialized irreducibility statement. The remaining sections are devoted to the proof of that reduced theorem.

\section{Strong dichotomy} \label{Sec:2}
In this section, our Laurent polynomials do not need to come from dispersion of graphs, which may be of independent interest.
We recall the meaning of ``generic'' used throughout the paper. Let $X=\CC^k$ be an affine parameter space with coordinates $\xi_1,\dots,\xi_k$. A subset $Z\subset X$ is called \emph{Zariski closed} if there exist (finitely many) polynomials
\[
P_1,\dots,P_m\in \CC[\xi_1,\dots,\xi_k]
\]
such that
\[
Z=\{\xi\in \CC^k: P_1(\xi)=\cdots=P_m(\xi)=0\}.
\]
A subset $U\subset X$ is called \emph{Zariski open} if its complement $X\setminus U$ is Zariski closed.

We say that a property holds \emph{generically} on $\CC^k$ if it holds on a nonempty Zariski-open subset of $\CC^k$.

Denote by $\xi=(\xi_1,\xi_2,\dots,\xi_k)$, $z=(z_1,z_2,\dots, z_d)$ and $z^{\pm}=(z_1^\pm,z_2^\pm,\dots, z_d^\pm)$.

Consider the family
\[
p(\xi;z,\lambda)
=
\lambda^n+\sum_{j=0}^{n-1}\sum_{\ell\in\mathbb Z^d} f_j^\ell(\xi)\,z^\ell\lambda^j
\in
\mathbb C[\xi,z^\pm,\lambda],
\]
where all but finitely many of the polynomials
$f_j^\ell\in\mathbb C[\xi]$ are zero, and
\[
z^\ell=z_1^{\ell_1}\cdots z_d^{\ell_d},
\qquad
\ell=(\ell_1,\dots,\ell_d)\in\mathbb Z^d.
\]
For $\xi\in\mathbb C^k$, we write
\[
p_\xi(z,\lambda)=p(\xi;z,\lambda).
\]

\begin{Theorem}\label{dich}
Let
\[
\mathcal R
:=
\{\xi\in\mathbb C^k : p_\xi(z,\lambda)\text{ is reducible in }\CC[z^\pm,\lambda]\}.
\]
Then $\mathcal R$ is Zariski closed in $\mathbb C^k$. Consequently, exactly one
of the following   holds:
\begin{enumerate}
\item $p_\xi$ is reducible for every $\xi\in\mathbb C^k$;
\item for generic $\xi\in \mathbb C^k$, $p_\xi$ is irreducible.
\end{enumerate}
\end{Theorem}

We prove Theorem~\ref{dich} through a sequence of lemmas and a proposition.

For a nonzero Laurent polynomial
\[
F(z,\lambda)=\sum_{(\alpha,i)} c_{\alpha,i} z^\alpha \lambda^i\in \CC[z^\pm,\lambda],
\]
we denote by $\operatorname{Supp}(F)\subset \mathbb Z^d\times\mathbb Z_{\ge 0}$
its support and by
\[
N(F)=\operatorname{conv}(\operatorname{Supp}(F))
\subset \mathbb R^d\times\mathbb R
\]
its Newton polytope. We shall use the standard identity
\[
N(FG)=N(F)+N(G)
\]
for nonzero $F,G\in \CC[z^\pm,\lambda]$, where $+$ one the right-hand side is  the
Minkowski sum.

Clearly, since $p_\xi$ is monic, we have the following lemma.
\begin{Lemma}\label{lem:monic_normalization}
Suppose that $p_\xi$ is reducible in $\CC[z^\pm,\lambda]$. Then there exists an integer
$r$ with $1\le r\le n-1$ and a factorization
\[
p_\xi=hg
\]
such that $h,g\in \CC[z^\pm,\lambda]$ are monic in $\lambda$ and
\[
\deg_\lambda h=r,
\qquad
\deg_\lambda g=n-r.
\]
\end{Lemma}

\begin{Lemma}\label{lem:support_bound}
Let
\[
S
:=
\{(\mathbf 0,n)\}
\cup
\{(\ell,j)\in\mathbb Z^d\times\{0,\dots,n-1\}: f_j^\ell\not\equiv 0\},
\]
where $\mathbf 0$ denotes the zero vector in $\mathbb Z^d$, and define
\[
N:=\operatorname{conv}(S)\subset \mathbb R^d\times\mathbb R.
\]
For each $r$ with $1\le r\le n-1$, let
\[
A_r
=
\bigl(N-(\mathbf 0,n-r)\bigr)\cap
\bigl(\mathbb Z^d\times\{0,\dots,r\}\bigr),
\]
\[
B_r
=
\bigl(N-(\mathbf 0,r)\bigr)\cap
\bigl(\mathbb Z^d\times\{0,\dots,n-r\}\bigr).
\]
Then $A_r$ and $B_r$ are finite. Moreover, if
\[
p_\xi=hg
\]
with $h,g\in \CC[z^\pm, \lambda]$ monic and
\[
\deg_\lambda h=r,
\qquad
\deg_\lambda g=n-r,
\]
then
\[
\operatorname{Supp}(h)\subset A_r
\text{ and }
\operatorname{Supp}(g)\subset B_r.
\]
\end{Lemma}

\begin{proof}
Since $N$ is a polytope, the sets $A_r$ and $B_r$ are finite.

For every specialization $\xi\in\mathbb C^k$, the support of $p_\xi$ is contained
in $S$, hence
\[
N(p_\xi)\subset N.
\]
Now assume that
\[
p_\xi=hg
\]
with $h,g$ monic of $\lambda$-degrees $r$ and $n-r$. Then
\[
N(h)+N(g)=N(hg)=N(p_\xi)\subset N.
\]
Since $g$ is monic of degree $n-r$, the point $(\mathbf 0,n-r)$ belongs to $N(g)$. Therefore
\[
N(h)+(\mathbf 0,n-r)\subset N(h)+N(g)\subset N,
\]
so
\begin{equation}\label{gapr191}
N(h)\subset N-(\mathbf 0,n-r).
\end{equation}

Because $\deg_\lambda h=r$, we also have
\begin{equation}\label{gapr192}
\operatorname{Supp}(h)\subset \mathbb Z^d\times\{0,\dots,r\}.\end{equation}
 
By \eqref{gapr191}  and \eqref{gapr192},  one has 
\[
\operatorname{Supp}(h)\subset A_r.
\]

The proof for $\operatorname{Supp}(g)\subset B_r$ is similar.
\end{proof}

\begin{Proposition}\label{prop:closed_degree_split}
For each integer $r$ with $1\le r\le n-1$, the set
\[
\mathcal R_r
=
\Bigl\{
\xi\in\mathbb C^k :
p_\xi \text{ admits a factorization } p_\xi=fg
\text{ with }
\deg_\lambda f=r,\ \deg_\lambda g=n-r
\Bigr\}
\]
is Zariski closed in $\mathbb C^k$.
\end{Proposition}

\begin{proof}
Fix $r$. Let $V_r$ be the finite-dimensional $\mathbb C$-vector space with
basis
\[
\{z^\alpha\lambda^i : (\alpha,i)\in A_r\},
\]
and let $W_r$ be the finite-dimensional $\mathbb C$-vector space with basis
\[
\{z^\beta\lambda^j : (\beta,j)\in B_r\}.
\]
Let
\[
C_r=S\cup (A_r+B_r),
\qquad
A_r+B_r=\{u+v : u\in A_r,\ v\in B_r\},
\]
and let $E_r$ be the finite-dimensional $\mathbb C$-vector space with basis
\[
\{z^\gamma\lambda^m : (\gamma,m)\in C_r\}.
\]

Multiplication defines a bilinear map
\[
V_r\times W_r\longrightarrow E_r.
\]
Clearly, the bilinear map induces a morphism
\[
m_r:\mathbb P(V_r)\times \mathbb P(W_r)\longrightarrow \mathbb P(E_r),
\qquad
([F],[G])\longmapsto [FG].
\]
Because $\mathbb P(V_r)\times\mathbb P(W_r)$ is projective, its image
\[
Y_r=m_r\bigl(\mathbb P(V_r)\times \mathbb P(W_r)\bigr)
\]
is Zariski closed in $\mathbb P(E_r)$.

On the other hand, the support of $p_\xi$ is contained in $S\subset C_r$ for all
$\xi\in\mathbb C^k$. Since the coefficient of $\lambda^n$ in $p_\xi$ is always $1$,
the coefficient vector of $p_\xi$ is never zero. Therefore  the following map is well defined: 
\[
\Phi_r:\mathbb C^k\longrightarrow \mathbb P(E_r),
\qquad
\xi\longmapsto [p_\xi].
\]

By  Lemma~\ref{lem:support_bound},  
\[
\mathcal R_r=\Phi_r^{-1}(Y_r).
\]

Hence $\mathcal R_r$ is Zariski closed.
\end{proof}

\begin{proof}[\bf Proof of Theorem~\ref{dich}]
By Lemma~\ref{lem:monic_normalization}, every nontrivial factorization of $p_\xi$
determines an integer $r$ with $1\le r\le n-1$ such that $\xi\in\mathcal R_r$.
Hence
\[
\mathcal R=\bigcup_{r=1}^{n-1}\mathcal R_r.
\]
By Proposition~\ref{prop:closed_degree_split}, each $\mathcal R_r$ is Zariski
closed in $\mathbb C^k$. Therefore $\mathcal R$ is Zariski closed as a finite
union of closed sets.
So, either
$\mathcal R=\mathbb C^k$, or $\mathcal R$ is a proper closed subset. In the
latter case,
\[
U=\mathbb C^k\setminus \mathcal R
\]
is a nonempty Zariski-open subset, and by definition $p_\xi$ is irreducible for
every $\xi\in U$.
\end{proof}

\begin{Corollary}\label{cor:one_point}
If there exists $\xi_0\in\mathbb C^k$ such that $p_{\xi_0}$ is irreducible in
$\CC[z^\pm,\lambda]$, then $p_\xi$ is irreducible for generic $\xi\in\mathbb C^k$.
\end{Corollary}

\begin{proof}
If $p_{\xi_0}$ is irreducible, then $\xi_0\notin\mathcal R$, so
$\mathcal R\neq\mathbb C^k$. The conclusion now follows from
Theorem~\ref{dich}.
\end{proof}

 \section{ Basics}~\label{Sec:MainRes}
 
\subsection{Periodic Graphs and Discrete Periodic Operators}
A graph $\Gamma = (\mathcal V, \mathcal E) $ is said to be  $\mathbb{Z}^d $-periodic if:

	1.	There exists a free action of  $\mathbb{Z}^d$  on  $\Gamma$, meaning no nonzero element of  $\mathbb{Z}^d $ fixes any vertex.
    
	2.	The action of $\mathbb{Z}^d $ is invariant, i.e., for every edge  $(u, v) \in \mathcal E $ and every $a \in \mathbb{Z}^d$, we have $( u+a,  v+a) \in \mathcal E $.
    
	3.	The action is cocompact, meaning that the quotient graph  $
    \Gamma/ \mathbb{Z}^d $ is a finite graph.

 For any $\ZZ^d$-periodic graph $\Gamma$,  a finite set  $W \subset \calV(\Gamma)$ is called 
 fundamental domain if   $\bigcup_{a \in \ZZ^d} W+a = \calV(\Gamma)$ and $|W| = |\calV(\Gamma)/\ZZ^d|$.

When there is no ambiguity,  we will use $n$ in place of $|W|$, and label the vertices of $W$ by the elements of ${[n] := \{ 1,\dots, n\}}$.

A {$\ZZ^d$-periodic edge labeling} $E : \calE(\Gamma) \to \CC$ is a periodic and symmetric function:  for any $(u,v) \in \calE(\Gamma)$, $$E((u,v))=E((v,u))$$ and for all $a \in \ZZ^d$, $$E((u,v)) = E((u+a,v+a)).$$
We will often call the value $E((u,v))$ the  {label}/weight of the edge $(u,v)$. 

An edge labeling on $\Gamma$ gives rise to a  {labeled adjacency operator $A_E$}. $A_E$ acts on a function $f$ on $\calV(\Gamma)$ as follows, \[A_E(f)(u) = \sum_{v:(u,v) \in \calE(\Gamma)} E((u,v))f(v), u \in \calV(\Gamma).\] 

A {$\ZZ^d$-periodic potential} $V : \calV(\Gamma) \to \CC$ is a periodic function on $\calV(\Gamma)$ such that $V(u) = V(u+a)$ for all $u \in \calV(\Gamma)$ and $a \in \ZZ^d$. We call a pair of functions $(V,E)$ a {labeling} of $\Gamma$.

Given a labeling $(V,E)$ of $\Gamma$, the operator ${\Lc} := V + A_E$, acts on a function $f$ on $\calV(\Gamma)$ as follows:

\[\Lc(f)(u) = V(u) f(u) + \sum_{v:(u,v) \in \calE(\Gamma)} E((u,v))f(v), u \in \calV(\Gamma).\]

\subsection{Floquet Theory}

Fix a vertex fundamental domain $W$ of $\Gamma$, and let $$\TT^d := \{(u_1,\dots, u_d) \in \CC^d \mid  |u_i| = 1 \text { for each } i  \}.$$ The Floquet transform $\mathscr{F}$ of a function $f$ on $\calV(\Gamma)$ is given by: \[ f(u) \mapsto \hat{f}(z,u) = \sum_{a \in \ZZ^d} f(u+a)z^{-a}.\]

Notice that $\hat{f}(z,u+a) = z^a \hat{f}(z,u)$. One can see that $\hat{f}(z,u)$ is just the Fourier transform of $f$ restricted to the orbit $u + \ZZ^d \subseteq \calV(\Gamma)$. If $f \in \ell^2(\calV(\Gamma))$,  then by Parseval's identity,  we have
\begin{equation}\label{gmar22}
    \sum_{u \in W} \int_{\TT^d} \vert \hat{f}(z,u) \vert^2  dz=\sum_{v\in \mathcal V(\Gamma)} |f(v)|^2.
\end{equation}

Thus, we can view $\hat{f}(z,\cdot) = (\hat{f}(z,1), \hat{f}(z,2), \dots, \hat{f}(z,n))^T$ as an element of the Hilbert space $L^2(\TT^d,\CC^n)$. By \eqref{gmar22}, $\mathscr{F}$ is a unitary operator; so, $\Lc$ and $\mathscr{F} \Lc \mathscr{F}^*$ are unitarily equivalent.

Direct calculation implies that, for $u \in \calV(\Gamma)$,
\begin{equation}\label{gflo1}
    \mathscr{F} \Lc \mathscr{F}^*(\hat{f})(z,u) = V(u) \hat{f}(z,u) + \sum_{\substack{w\in W,\ a\in\ZZ^d\\ (u,w+a)\in \calE(\Gamma)}} E((u,w+a))z^a \hat{f}(z,w).
\end{equation}

For each $z \in \TT^d$, let $L(z)$ be a $|W|\times |W|$ matrix   (with rows and columns indexed by the vertices of $W$) given by the following: for $u,v \in W$,
\begin{equation}\label{gflo2}
    L(z)_{u,v} = \delta_{u,v}V(u)+\sum_{a\in\ZZ^d:(u,v+a)\in \calE(\Gamma)}E((u,v+a))z^{a},
\end{equation}
 where $\delta$ is the Kronecker delta function ($\delta_{u,v} = 1$ when $u = v$ and is $0$ otherwise). 
$L(z)$ is called the Floquet matrix.
By \eqref{gflo1} and \eqref{gflo2}, $ \mathscr{F} \Lc \mathscr{F}^*$ is the direct integral of $L(z)$:
\begin{equation}~\label{eq:decomposition} \mathscr{F} \Lc \mathscr{F}^* = \int_{\TT^d}^{\oplus} L(z) \ dz.\end{equation}

 We remark that each entry of $L(z)$ is a Laurent polynomial in $z$, and 
 \begin{equation}\label{gfeb252}
   L(z) = L(z^{-1})^T.
\end{equation}

\subsection{Dispersion polynomials}
 We call the characteristic polynomial of $L(z)$ the \emph{dispersion polynomial}, and denote it by
\begin{equation}\label{eq:dispersionPoly}
    D(z,\lambda)=\det(L(z)-\lambda I).
\end{equation}
The \emph{dispersion relation}, or \emph{Bloch variety}, is the zero set of $D(z,\lambda)$ in $(\CC^*)^d\times \CC$.

It is more common to describe the Bloch variety using quasimomentum variables $k=(k_1,\dots,k_d)$, where
\[
z_j=e^{2\pi i k_j}.
\]
In this paper we work with the variables $z$ instead. By a slight abuse of terminology, we will still refer to the zero set of $D(z,\lambda)$ as the Bloch variety.

\subsection{ Quotient Graphs}
 
\begin{Definition}
We define an equivalence relation on $\calV(\Gamma)$:
\begin{equation}\label{gapr181}
    u\sim v
\qquad\text{if and only if}\qquad
u=v+a
\quad\text{for some }a\in \ZZ^d.
\end{equation}

We write $[u]$ for the equivalence class of $u$.

Let $ {S(\Gamma/\ZZ^d)}$ be the finite simple graph with vertex set
\[
\calV(S(\Gamma/\ZZ^d)):=\calV(\Gamma)/\ZZ^d
\]
and edge set
\[
\calE(S(\Gamma/\ZZ^d))
=
\{([u],[v])\mid [u]\neq [v]\text{ and }(u,v+a)\in \calE(\Gamma)\text{ for some }a\in \ZZ^d\}.
\]
We call $S(\Gamma/\ZZ^d)$ the  quotient graph of $\Gamma$.  
\end{Definition}

\subsection{Resultant}
Given two polynomials $f(\lambda) = \sum_{i=0}^s a_i \lambda^i$ and $g(\lambda)=\sum_{i=0}^t b_i \lambda^i$ of degree $s$ and $t$ respectively, the {Sylvester matrix} is the $(s+t) \times (s+t)$ matrix
{\small\[ \begin{pmatrix}
    a_0 & a_1 & a_2 & \dots & a_{s-1} & a_s & 0 & 0 & \dots &0 \\
    0 & a_0 & a_1 & \dots & a_{s-2} & a_{s-1} & a_s & 0 & \dots & 0 \\
    \vdots & \vdots & \ddots & \ddots & \ddots& \ddots& \ddots& \ddots & \ddots & \vdots \\
    0 & 0 & 0 & \dots & a_{0} & a_{1} & a_2 & a_3 & \dots & a_s \\
    b_0 & b_1 & b_2 & \dots & b_{t-1} & b_t & 0 & 0 & \dots &0 \\
    0 & b_0 & b_1 & \dots & b_{t-2} & b_{t-1} & b_t & 0 & \dots & 0 \\
    \vdots & \vdots & \ddots & \ddots & \ddots& \ddots& \ddots& \ddots & \ddots & \vdots \\
    0 & 0 & 0 & \dots & b_{0} & b_{1} & b_2 & b_3 & \dots & b_t \\
\end{pmatrix}.\]}
The determinant of the Sylvester matrix is the (classical) {resultant}, denoted by ${\rm Res}(f,g)$.

\section{Main theorem reduction}\label{Sec:reduce}
 
In this section,  we are going to reduce Theorem \ref{Thm:main}  to   
\begin{Theorem}~\label{Thm:main2}
Assume that  $\Gamma$ has a fundamental domain $W$
which is minimally connected and there exists exactly one edge between $W$ and $\Gamma\backslash W$.

After fixing non-zero edge weights,  $D(x,\lambda)$ is irreducible for a generic choice of potential.
\end{Theorem}

\begin{Lemma}\label{keyliu1}
The quotient graph of $\Gamma$ is connected if and only if there exists a fundamental domain $W$ such that $W$ is connected (as   vertices in $\Gamma$).
\end{Lemma}

\begin{proof}
It is clear that if there exists a fundamental domain $W$ such that $W$ is connected, then the quotient graph of $\Gamma$ is connected.

Now assume that the quotient graph of $\Gamma$ is connected. We are going to  construct a connected fundamental domain by induction.

Choose one equivalence class (see \eqref{gapr181}) and pick a representative $x_1$. Suppose we have already chosen representatives
\[
x_1,\dots,x_m
\]
from equivalence classes such that these vertices are connected in $\Gamma$.

Since the quotient graph is connected, if not all equivalence classes have been chosen, then there exists another equivalence class $[y_{m+1}]$ which is connected  to one of 
\[
[x_1],\dots,[x_m].
\]
Without loss of generality, assume that $[y_{m+1}]$ is connected  to $[x_1]$. Then there exist $a,b\in\ZZ^d$ such that
\[
x_1+a \quad \text{and} \quad y_{m+1}+b
\]
are connected in $\Gamma$. Let
\[
x_{m+1}:=y_{m+1}+b-a.
\]
Then $x_{m+1}$ is connected to $x_1$. Hence the vertices
\[
x_1,\dots,x_m,x_{m+1}
\]
are  connected in $\Gamma$.

Continuing this process, we obtain one representative from each equivalence class, and these representatives form a connected fundamental domain.
\end{proof}

\begin{proof}[\bf Proof of Theorem~\ref{Thm:main} based on Theorem~\ref{Thm:main2}]
Assume first that the quotient graph of $\Gamma$ is connected. By Lemma~\ref{keyliu1}, there exists a connected fundamental domain $W$. Since $\Gamma$ is nontrivial, there exists at least one edge between $W$ and $\Gamma\backslash W$. By specializing enough edge weights (with respect to $\ZZ^d$) to zero, we may assume that $W$ is minimally connected and that there is exactly one edge between $W$ and $\Gamma\backslash W$. Denote this edge by
\[
(u,v+a),
\qquad u,v\in W,\ a=(a_1,\dots,a_d)\in \ZZ^d.
\]
Without loss of generality, assume that $a_1\neq 0$. Set
\[
x=z_1,\qquad z_2=\cdots=z_d=1.
\]
After these specializations, we obtain a new graph $\Gamma'$ ($\ZZ$-periodic) and a new dispersion polynomial $D'(x,\lambda)$. We notice that $\Gamma'$ and $D'(x,\lambda)$ are obtained from $\Gamma$ and $D(z,\lambda)$ by specializing some edge weights to zero and setting $z_2=\cdots=z_d=1$.

Applying Theorem~\ref{Thm:main2} to $\Gamma'$, we have that $D'(x,\lambda)$ is irreducible for a generic choice of potential. In particular, there exists one choice of potential for which $D'(x,\lambda)$ is irreducible. For the corresponding choice of parameters, we claim that $D(z,\lambda)$ is irreducible. Otherwise, assume $D(z,\lambda)$ is  reducible, then its factorization would lead to a factorization of $D'(x,\lambda)$, which is  a contradiction. Therefore, by Corollary~\ref{cor:one_point}, $D(z,\lambda)$ is irreducible generically.

On the other hand, assume that $D(z,\lambda)$ is irreducible generically. Then the quotient graph of $\Gamma$ must be connected; otherwise the Floquet matrix would contain at least two blocks, and hence $D(z,\lambda)$ would be reducible. Also, $\Gamma$ must be nontrivial. Indeed, if $\Gamma$ is trivial, then for a suitable fundamental domain $W$,   $L(z)$  is  independent of $z$. Therefore, $D(z,\lambda)$ is reducible (recall our assumption $|W|\geq 2$). So $\Gamma$ is nontrivial.
\end{proof}
    The irreducibility of the Bloch variety, namely the zero set of the dispersion polynomial, is slightly different from the irreducibility of the dispersion polynomial itself because of the multiplicities. We first give two   lemmas before proving Corollary~\ref{maincor}.

    \begin{Lemma}\label{lemsimple}
Let $B\in M_n(\CC)$ be fixed. For $v=(v_1,\dots,v_n)\in \CC^n$, define
\[
M(v)=\operatorname{diag}(v_1,\dots,v_n)+B.
\]
Then the set
\[
\calU_{\mathrm{simp}}(B)
=
\{v\in \CC^n : M(v)  \text{ has } n \text{ simple eigenvalues} \}
\]
is a nonempty Zariski-open subset of $\CC^n$.
\end{Lemma}

\begin{proof}

Let $f_v(\lambda)=\det (M(v)-\lambda I)$. ${\rm Res} (f_v(\lambda), f_v'(\lambda))$ is a polynomial of $v$ and ${\rm Res} (f_v(\lambda), f_v'(\lambda))\neq 0$ iff $M(v)$    has  $n$   simple eigenvalues.

 By choosing the potentials
$v_1,\dots,v_n$ pairwise sufficiently far apart, the Gershgorin discs of $M(v)$ become pairwise disjoint. It implies that  $M(v)$ has $n$ simple eigenvalues for such a choice of potential. Therefore the ${\rm Res} (f_v(\lambda), f_v'(\lambda))$
is not identically zero as a polynomial in $(v_1,\dots,v_n)$, and so the set
$\calU_{\mathrm{simp}}(B)$
is a nonempty Zariski-open subset of $\CC^n$.
\end{proof}
\begin{Lemma}\label{lem:BV-to-poly}
Let
\[
f(z,\lambda)\in \CC[z_1^{\pm1},\dots,z_d^{\pm1},\lambda]
\]
be monic in $\lambda$, and let
\[
\mathcal B_f=\{(z,\lambda)\in (\CC^*)^d\times \CC : f(z,\lambda)=0\}.
\]
Assume that $\mathcal B_f$ is irreducible. If there exists $z^0\in (\CC^*)^d$ such that
$f(z^0,\lambda)$ has only simple roots in $\lambda$, then $f$ is irreducible in
$\CC[z_1^{\pm1},\dots,z_d^{\pm1},\lambda].$
 
\end{Lemma}

\begin{proof}
We are going to prove it by contradiction.
Assume $f$ is reducible in $\CC[z_1^{\pm1},\dots,z_d^{\pm1},\lambda]$.
Since the zero set of $f$ is irreducible, we must have that 
\[
f=h^m
\]
for some irreducible polynomial $h$ that is monic in $\lambda$ and $m\geq 2$.

Now consider $z=z^0$. We  have that 
\[
f(z^0,\lambda)=h(z^0,\lambda)^m
\]
has repeated roots. This contradicts the assumption that $f(z^0,\lambda)$ has only simple roots. 
\end{proof}
\begin{proof}[\bf Proof of Corollary \ref{maincor}]
Assume first that the quotient graph $\Gamma$ is connected and that $\Gamma$ is nontrivial. By Theorem~\ref{Thm:main}, the dispersion polynomial
$D(z,\lambda)$
is generically irreducible  and hence its zero set
\[
\{(z,\lambda)\in (\CC^*)^d\times \CC : D(z,\lambda)=0\}
\]
is irreducible. Thus the Bloch variety is generically irreducible.

Conversely, assume that the Bloch variety is generically irreducible. By Lemma \ref{lemsimple}, 
generically $D(1,\lambda)$ has only simple roots.
 Applying Lemma~\ref{lem:BV-to-poly} to
$
f(z,\lambda)=D(z,\lambda),
$
we conclude that generically $D(z,\lambda)$ is irreducible. Theorem~\ref{Thm:main} now implies that the quotient graph $\Gamma$ is connected and $\Gamma$ is nontrivial.
\end{proof}
By Lemmas~\ref{lemsimple} and~\ref{lem:BV-to-poly}, and by a slight modification of the proof of Corollary~\ref{maincor}, we obtain the following refinement:

\begin{Theorem}\label{thm:components-factors}
Fix the edge weights for a $\ZZ^d$ periodic graph. For a generic choice of potentials, every irreducible factor of the dispersion polynomial $D(z,\lambda)$ in
\[
\CC[z_1^{\pm1},\dots,z_d^{\pm1},\lambda]
\]
has multiplicity one.

More precisely, let
\[
\widetilde D(z,\lambda)=(-1)^nD(z,\lambda),
\]
so that $\widetilde D$ is monic in $\lambda$. For a generic choice of potentials, if
\[
\widetilde D(z,\lambda)=D_1(z,\lambda)\cdots D_m(z,\lambda)
\]
is the factorization of $\widetilde D$ into irreducible factors, with each $D_j$ chosen monic in $\lambda$, then the factors $D_1,\dots,D_m$ are pairwise distinct.
\end{Theorem}
\section{Technical preparations}~\label{Sec:tech}

Throughout this section we work in the setting of Theorem~\ref{Thm:main2}. Thus $\Gamma$ is a $\ZZ$-periodic graph with a minimally connected fundamental domain $W=[n]$, and there is exactly one edge between $W$ and $\Gamma\backslash W$. Fix the edge (nonzero) weights, write $L(x)$ for the Floquet matrix, and set
\[
D(x,\lambda)=\det(L(x)-\lambda I).
\]
Because there is exactly one edge between $W$ and $\Gamma\backslash W$, there exists a single pair of vertices $u,v\in W$ and exactly one nonzero integer $a$ such that $u$ and $v+a$ are connected. Therefore, the only $x$-dependent terms in $L(x)$ are the monomials $x^a$ and $x^{-a}$ coming from the edge $(u,v+a)$ and its reverse.

In order to prove Theorem~\ref{Thm:main2}, we analyze the determinant expansion of $D(x,\lambda)$.

For a Laurent polynomial,
\[
f(\xi)=\sum_{b\in\ZZ^k} c_b \xi^b,
\]
we write
\[
[\xi^b]f:=c_b
\]
for the coefficient of the monomial $\xi^b$ in $f$.

Let ${S_n}$ denote the symmetric group on $[n]$. For $\sigma\in S_n$, define
\begin{equation}\label{eq:perm}
{\sigma D(x,\lambda)}=\prod_{i=1}^n (L(x)-\lambda I)_{i,\sigma(i)}.
\end{equation}
Similarly, if $\eta=(j_1\dots j_p)$ is a cycle, then we write
\begin{equation}\label{eq:cycle}
{\eta D(x,\lambda)}=\prod_{i=1}^p (L(x)-\lambda I)_{j_i,\eta(j_i)}.
\end{equation}
If $\sigma=\eta_1\cdots \eta_m$ is the decomposition of $\sigma$ into disjoint cycles, then
\[
\sigma D(x,\lambda)=\prod_{r=1}^m \eta_r D(x,\lambda).
\]

By the definition of determinant,
\begin{equation}\label{defdet}
D(x,\lambda)=\sum_{\sigma\in S_n} \operatorname{sgn}(\sigma)\,\sigma D(x,\lambda).
\end{equation}

\begin{Lemma}\label{lem1}
There is a unique cycle $\eta$ such that $[x^a]\eta D(x,\lambda)$ is nontrivial.
\end{Lemma}

\begin{proof}
Let the unique edge between $W$ and $\Gamma\backslash W$ be $(u,v+a)$, where $u,v\in W$ and $a\neq 0$. Since $W$ is minimally connected, there is a unique simple path in $W$ from $v$ to $u$. Together with the edge $(u,v+a)$, this path determines a unique cycle $\eta$ in the quotient graph such that $[x^a]\eta D(x,\lambda)\neq 0$.

\end{proof}
Denote by $p$ the size of this cycle $\eta$. Without loss of generality, assume that
\begin{equation}\label{g12}
\eta=(1\cdots p).
\end{equation}
Let
\[
U=\{p+1,p+2,\dots,n\}.
\]
For $1\le j\le p-1$, write $e_{j,j+1}$ for the weight of the  edge $(j,j+1)$, and write $e_{(p,1),a}$ for the weight of the unique   edge $(p,1+a)$. Without loss of generality, assume $a>0$.

When $p=n$, $U$ is $\emptyset$. In this case, we define $\det (L|_{U}-\lambda I_{n-p})=1$.

\begin{Proposition}\label{prop}
	The following holds.
	\begin{enumerate}
		\item[A.]
		\begin{equation}
[x^a] D(x,\lambda)= \det (L|_{U}-\lambda I_{n-p})\,(-1)^{p-1} e_{(p,1),a}\prod_{j=1}^{p-1} e_{j, j+1},
		\end{equation}
		where $I_{n-p}$ is the identity matrix of size $n-p$.
		\item[B.]
	Let
\[
s(\lambda)=\det (L|_{U}-\lambda I_{n-p})\,(-1)^{p-1} e_{(p,1),a}\prod_{j=1}^{p-1} e_{j, j+1}.
\]
	Then there exists a polynomial $ r(\lambda) \in \CC[\lambda]$ such that
	\begin{equation}\label{keyDet}
D(x,\lambda)=r(\lambda) +s(\lambda) x^a+s(\lambda) x^{-a}.
	\end{equation}
		\item[C.]
	Polynomials $ r(\lambda)\in \CC[\lambda]$ and $s(\lambda) \in\CC[\lambda]$ satisfy
		\begin{itemize}
			\item $\deg r(\lambda)=n$ and $[\lambda^n] r(\lambda)=(-1)^n$;
			\item $\deg s(\lambda)=n-p$ and
\[
[x^a\lambda^{n-p}] D(x,\lambda)= [\lambda^{n-p}] s(\lambda)=(-1)^{n-1} e_{(p,1),a}\prod_{j=1}^{p-1} e_{j, j+1}.
\]
		\end{itemize}
	\end{enumerate}
\end{Proposition}

  \begin{proof}
By \eqref{eq:cycle} and \eqref{g12}, one has
\begin{equation}\label{g13}
[x^a] \eta D(x,\lambda)= e_{(p,1),a}\prod_{j=1}^{p-1} e_{j, j+1}.
\end{equation}
Now Part A  of Proposition \ref{prop} follows from Lemma \ref{lem1},  \eqref{g13} and \eqref{defdet}.

Parts B and C  of Proposition \ref{prop}  simply follow from \eqref{defdet} and Part A.

  \end{proof}

Let
\[
A_k=\{\mu\in \CC:\mu^k=1\},
\qquad
A_k^l=\left\{e^{\frac{2\pi i j}{k}}:j=0,\dots,l-1\right\}.
\]
We begin with a helper lemma.

\begin{Lemma}~\label{lemliu1}
Let
\[
F(x,\lambda)=r(\lambda)+s(\lambda)x+s(\lambda)x^{-1},
\]
where $r(\lambda)\neq 0$, the leading term of $r(\lambda)$ is $(-1)^n\lambda^n$, and $s(\lambda)\neq 0$ with $\deg s\le n-1$. Let $a$ be a positive integer.
If  $F(x^a,\lambda)$ is reducible, then $F(x^a,\lambda)$ falls into one of the following three cases:
\begin{enumerate}
    \item
    $F(x^a,\lambda)$ has  a  linear factor in $\lambda$.

    \item
    There exist $l\ge 2$ with $l\mid a$, and polynomials
$r_1(\lambda), s_1(\lambda), t_1(\lambda)\in \CC[\lambda]$, such that
\begin{equation}\label{liug1}
    F(x^a,\lambda)
    =
     \prod_{\mu\in A_l}
    \bigl(r_1(\lambda)+s_1(\lambda)\mu x^{a/l}+t_1(\lambda)\mu^{-1}x^{-a/l}\bigr),
\end{equation}
where
\[
r_1(\lambda)+s_1(\lambda)x^{a/l}+t_1(\lambda)x^{-a/l}
\]
is irreducible. Moreover, there exists $\xi\in A_l$ such that
\[
s_1(\lambda)=\xi\, t_1(\lambda).
\]
    \item
    There exist $l\ge 1$ with $l\mid a$, and polynomials $r_2(\lambda), s_2(\lambda)\in \CC[\lambda]$, such that
    \begin{equation}\label{liug2}
        F(x^a,\lambda)
        =
        \prod_{\mu\in A_l}
        \bigl(r_2(\lambda)+s_2(\lambda)\mu x^{a/l}\bigr)
        \bigl(r_2(\lambda)+s_2(\lambda)\mu x^{-a/l}\bigr),
    \end{equation}
    and $r_2(\lambda)+s_2(\lambda)x^{a/l}$ is irreducible. 
\end{enumerate}
\end{Lemma}
\begin{proof}
If $F(x^a,\lambda)$ has a linear factor in $\lambda$, then we are in Case 1. Thus assume that $F(x^a,\lambda)$  has no linear factors.

    Let $g(x,\lambda)$ be a irreducible factor of $ F(x^a,\lambda)$. 
    Multiplying $g$ by a Laurent monomial in $x$ if necessary, we may write
\begin{equation}\label{eq:g-normal-form}
    g(x,\lambda)=\sum_{j=b_1}^{b_2} g_j(\lambda)x^j,
\end{equation}
where $b_1\le 0\le b_2$, $g_0(\lambda) $ is monic and nonzero, and
\[
\deg_\lambda g_0(\lambda)\ge \deg_\lambda g_j(\lambda)
\qquad\text{for all }j.
\]

Take a nonzero monomial  $g_b(\lambda) x^b$ with  $b\neq 0$,  $b_1\leq b\leq b_2$ into consideration.
 Since
\[
F((\mu x)^a,\lambda)=F(x^a,\lambda)
\qquad\text{for every }\mu\in A_a,
\]
it follows that
\[
g(\mu x,\lambda)\mid F(x^a,\lambda)
\qquad\text{for every }\mu\in A_a.
\]
Moreover, because $g_0\neq 0$, we have that
if $\mu_1^b\neq \mu_2^b$, then 
 factors $g(\mu_1 x,\lambda)$ and $g(\mu_2 x,\lambda)$   are distinct.

 This implies   that there are at least $  \frac{a}{\gcd(|b|,a)}$ many distinct irreducible factors of $F(x^a,\lambda)$ of the form $g(\mu x,\lambda)$, $\mu\in A_a$. Note that there is one such factor for each $\mu \in A_a^{\frac{a}{\gcd(|b|,a)}}$.

We now  split the proof into two cases.

\smallskip
\noindent\textit{Case I: $g$ has only nonnegative or only nonpositive powers of $x$.}

 Without loss of generality,  we  assume that $g$ has only nonnegative powers of $x$.
 
 Take $b=b_2$. 
 The highest $x$-degree of
\[
\prod_{\mu\in A_a^{a/\gcd(a,b_2)}} g(\mu x,\lambda)
\]
is
\[
b_2\frac{a}{\gcd(a,b_2)}.
\]
Since this product divides $F(x^a,\lambda)$, whose highest power of $x$ is $x^a$, we must have
\[
b_2\frac{a}{\gcd(a,b_2)}\le a.
\]
This leads to $\gcd(a,b_2)=b_2$, i.e. $b_2\mid a$.

We are going to prove that
\[
g_j(\lambda)=0
\qquad\text{for all }j=1,\dots,b_2-1.
\]
Indeed, if $g_j\neq 0$ for some $1\le j\le b_2-1$, then $g(\mu x,\lambda)$, $\mu\in A_a^{a/\gcd(a,j)}$ are distinct factors, and hence 
\begin{equation}\label{gapr195}
    \prod_{\mu\in A_a^{a/\gcd(a,j)}} g(\mu x,\lambda)
\end{equation}
divides $F(x^a,\lambda)$.  Clearly, the highest $x$-degree in \eqref{gapr195} is
\[
b_2\frac{a}{\gcd(a,j)}>a,
\]
a contradiction. Therefore, we conclude that
\[
g(x,\lambda)=g_0(\lambda)+g_{b_2}(\lambda)x^{b_2}.
\]

Since $F(x^a,\lambda)=F(x^{-a},\lambda)$, the polynomial
\[
g(x^{-1},\lambda)=g_0(\lambda)+g_{b_2}(\lambda)x^{-b_2}
\]
also divides $F(x^a,\lambda)$. Let
\[
l=\frac{a}{b_2}.
\]
Since the product of the factors
\begin{equation}\label{gapr262}
    g(\mu x,\lambda),
\qquad
g(\mu x^{-1},\lambda),
\qquad
\mu\in A_a^l,
\end{equation}
already contribute to $x^a$ and $x^{-a}$, matching the extreme $x^\pm$-degrees of $F(x^a,\lambda)$,  we conclude that  $F(x^a,\lambda)$ has no other factors other than the ones in \eqref{gapr262}. 

It follows that 
\[
F(x^a,\lambda)
=
(-1)^n\prod_{\mu\in A_a^l}
\bigl(g_0(\lambda)+g_{b_2}(\lambda)\mu^{b_2}x^{b_2}\bigr)
\bigl(g_0(\lambda)+g_{b_2}(\lambda)\mu^{b_2}x^{-b_2}\bigr).
\]
Since
\[
\{\mu^{b_2}:\mu\in A_a^l\}=A_l,
\]
this gives \eqref{liug2} after setting
\[
r_2(\lambda):=(-1)^{n/2l}g_0(\lambda),
\qquad
s_2(\lambda):=(-1)^{n/2l}g_{b_2}(\lambda).
\]

This is  Case 3.

Case II: $g$ has terms of both negative and positive degree in $x$.

Without loss of generality, assume $b_2\geq |b_1|$.

Exactly as above, the product
\[
\prod_{\mu\in A_a^{a/\gcd(a,b_2)}} g(\mu x,\lambda)
\]
divides $F(x^a,\lambda)$, and its highest $x$-degree is
\[
b_2\frac{a}{\gcd(a,b_2)}.
\]
Since the highest $x$-degree of $F(x^a,\lambda)$ is $a$, we again conclude that
\[
b_2\mid a.
\]

We next claim that
\[
g_j(\lambda)=0
\qquad\text{for all }j\text{ with }0<|j|<b_2.
\]
Indeed, if $g_j\neq 0$ for some such $j$, then
\[
\prod_{\mu\in A_a^{a/\gcd(a,j)}} g(\mu x,\lambda)
\]
divides $F(x^a,\lambda)$, while its highest $x$-degree is
\[
b_2\frac{a}{\gcd(a,j)}>a,
\]
which is impossible. Therefore
\[
g(x,\lambda)=g_{-b_2}(\lambda)x^{-b_2}+g_0(\lambda)+g_{b_2}(\lambda)x^{b_2},
\]
with both $g_{-b_2}$ and $g_{b_2}$ nonzero (otherwise, it reduces to  Case I).

Let
\[
l=\frac{a}{b_2}.
\]
The orbit factors
\[
g(\mu x,\lambda),
\qquad
\mu\in A_a^l,
\]
are pairwise distinct, and their product already has $x^a$ and $x^{-a}$. Hence, as in Case~I,  there are  no remaining factor. Thus 
\begin{equation}\label{gapr263}
    F(x^a,\lambda)=(-1)^n\prod_{\mu\in A_a^l} g(\mu x,\lambda).
\end{equation}

Writing
\[
r_1(\lambda)=(-1)^{n/l}g_0(\lambda),
\qquad
s_1(\lambda)=(-1)^{n/l}g_{b_2}(\lambda),
\qquad
t_1(\lambda)=(-1)^{n/l}g_{-b_2}(\lambda)
\]
and using again that
\[
\{\mu^{b_2}:\mu\in A_a^l\}=A_l,
\]
we obtain the factorization \eqref{liug1}.

Since $F(\lambda,x^a)=F(\lambda,x^{-a})$ and $g(x^{-1},\lambda)$ is also a factor, by \eqref{gapr263}, one has that there exists $\mu\in A_a^l$ such that $g(x^{-1},\lambda)=g(\mu x,\lambda)$.

This implies that  there exists $\xi\in A_l$ such that
\[
s_1(\lambda)=\xi\, t_1(\lambda).
\]

This is  Case 2.
  
    \end{proof}
\smallskip

\section{No linear factors}
In this section, our goal is to prove
\begin{Theorem}\label{thmonecase1}
    Assume that  $\Gamma$ has a fundamental domain $W$
which is minimally connected and there exists exactly one edge between $W$ and $\Gamma\backslash W$.

After fixing non-zero edge weights,  
for a generic potential, the dispersion polynomial
$D(x,\lambda)$  of $\Gamma$  has no linear factors. 
\end{Theorem} 
For any subset $U\subset W$, we write
\[
D_U(z,\lambda)=\det(L|_U(z)-\lambda I).
\]
If $D_U(z,\lambda)$ is independent of $z$, we simply write $D_U(\lambda)$.

\begin{Theorem}[{\cite[Theorem 8.1]{FaustLiu2025RareFlatBands}}]\label{Thm:8.1}
Assume that  the  quotient graph of $\Gamma$ is minimally connected.

	Let $W$ be a fundamental domain.  Assume that exists a $W_1 \subset W$   with exactly one less vertex such that $L(z)|_{W_1}$ is independent of $z$. If all edge weights are nonzero,  then for a generic choice of potential, $D(z,\lambda)$ and $D_{W_1}(\lambda)$ cannot have a common linear factor.
\end{Theorem}

\smallskip
\begin{Lemma}\label{Lem:new71}
Let $\Gamma$ be a $\ZZ^d$-periodic graph, and let $U\subset W$ be a subset of vertices such that $D_U(\lambda)$ is independent of $z$. Suppose that, for a generic choice of potential, $D(z,\lambda)$ and $D_U(\lambda)$ share a linear factor. If $j$ is a vertex in $W\setminus U$, then $D_{W\setminus\{j\}}(z,\lambda)$ and $D_U(\lambda)$ also share a linear factor for every potential.
\end{Lemma}
\begin{proof}
Clearly, $D(z,\lambda)$ and $D_U(\lambda)$ share a linear factor if and only if
\begin{equation}\label{gres}
    \operatorname{Res}(D,D_U)=0.
\end{equation}
Since the resultant is a polynomial in the potential variables, if \eqref{gres}
holds for generic potentials, then it holds identically for all potentials.

For any $j\in W\setminus U$, write
\begin{equation}\label{gres1}
    D(z,\lambda)=v_j D_{W\setminus\{j\}}(z,\lambda)+h(z,\lambda),
\end{equation}
where $h$ is independent of $v_j$.
Then
\[
\operatorname{Res}(D,D_U)
=
\operatorname{Res}(v_jD_{W\setminus\{j\}}+h,D_U)
\]
is a polynomial in $v_j$. Therefore, for $\operatorname{Res}(D,D_U)$, the highest-degree term in $v_j$ is 
\begin{equation}\label{gapr173}
(-1)^{|U|}\operatorname{Res}(D_{W\setminus\{j\}},D_U).
\end{equation}

By \eqref{gres} and \eqref{gapr173},  one has that
\[
\operatorname{Res}\bigl(D_{W\setminus\{j\}},D_U  \bigr)=0.
\]
Hence for every potential, $D_{W\setminus\{j\}}(z,\lambda)$ and $D_U(\lambda)$ share a linear factor.
\end{proof}

\begin{proof}[\bf Proof of Theorem~\ref{thmonecase1}]
Recall that
\[
D(x,\lambda)=r(\lambda)+(x^a+x^{-a})s(\lambda).
\]
Clearly, $D(x,\lambda)$ has a linear factor if and only if $r(\lambda)$ (also $D(x,\lambda)$) and $s(\lambda)$ have a common factor. In particular,
\begin{equation}\label{gapr176}
\{v\in \CC^n : D(x,\lambda)\text{ has a linear factor}\}
=
\{v\in \CC^n:\operatorname{Res}(D(x,\lambda),s(\lambda))=0\}.
\end{equation}
Hence the set of potentials for which $D(x,\lambda)$ has a linear factor is Zariski closed.

We are going to prove Theorem~\ref{thmonecase1} by contradiction. Assume that Theorem~\ref{thmonecase1} does not hold. Since the set in \eqref{gapr176} is Zariski closed, it follows that it must equal all of $\CC^n$. Therefore, for every potential, $D(x,\lambda)$ and $s(\lambda)$ share a linear factor.

Let
\[
U=\{p+1,p+2,\dots,n\}\subset W.
\]
If $U=\varnothing$, then by Proposition~\ref{prop}, $s(\lambda)$ is a nonzero constant, and hence $D(x,\lambda)$ cannot have a linear factor. Thus we  only need to consider the case $U\neq\varnothing$.

Let $U_1,\dots,U_k$ be the connected components of the graph induced by $U$. By Part~B of Proposition~\ref{prop},
\begin{equation}\label{g42}
s(\lambda)=\kappa\,D_U(\lambda)=\kappa\prod_{i=1}^k D_{U_i}(\lambda),
\end{equation}
where
\begin{equation}\label{gapr1910}
\kappa=(-1)^{p-1}e_{(p,1),a}\prod_{j=1}^{p-1}e_{j,j+1}\neq 0.
\end{equation}

For each $i=1,\dots,k$, let
\[
\mathcal R_i=\{v\in \CC^n:\operatorname{Res}(D(x,\lambda),D_{U_i}(\lambda))=0\}.
\]
Each $\mathcal R_i$ is Zariski closed. Moreover, by \eqref{g42}, if $D(x,\lambda)$ and $s(\lambda)$ share a linear factor, then $D(x,\lambda)$ must share a linear factor with at least one of the polynomials $D_{U_i}(\lambda)$. Therefore
\[
\CC^n\subset \bigcup_{i=1}^k \mathcal R_i.
\]
Therefore, one of the sets $\mathcal R_i$ must equal $\CC^n$. Without loss of generality, we may assume that $\mathcal R_1=\CC^d$ and hence for all potentials,
\begin{equation}\label{g43}
\operatorname{Res}(D(x,\lambda),D_{U_1}(\lambda))=0.
\end{equation}

Let $j_0\in\{1,2,\dots,p\}$ be the unique vertex such that $j_0$ is connected to $U_1$. Notice that, one by one, we can remove vertices in
\[
W\setminus (\{j_0\}\cup U_1),
\]
each time invoking Lemma~\ref{Lem:new71}. We conclude that
\[
D_{\{j_0\}\cup U_1}(x,\lambda)
\quad\text{and}\quad
D_{U_1}(\lambda)
\]
share a linear factor for every potential.

Now $\Gamma_{\{j_0\}\cup U_1}$ is a graph whose quotient graph is minimally connected, and  satisfies the other assumptions of Theorem~\ref{Thm:8.1}. Therefore Theorem~\ref{Thm:8.1} gives a contradiction.
\end{proof}
\section{Excluding  Case 2 of Lemma~\ref{lemliu1}} 
\begin{Theorem}\label{thmonecase2}
    Assume that  $\Gamma$ has a fundamental domain $W$
which is minimally connected and there exists exactly one edge between $W$ and $\Gamma\backslash W$.

After fixing non-zero edge weights,  
for a generic potential, the dispersion polynomial
$D(x,\lambda)$    can not be factorized in Case 2 appearing in Lemma \ref{lemliu1}. 
\end{Theorem}

\begin{proof}

In this section we rule out  Case 2 of Lemma~\ref{lemliu1}.  
Assume that  for a   choice of potential, $D(x,\lambda)$ admits a factorization as in Case 2 of Lemma~\ref{lemliu1}.

Recall that
\[
D(x,\lambda)=r(\lambda)+(x^a+x^{-a})s(\lambda),
\]
where
\[
\deg r(\lambda)=n,
\qquad
[\lambda^n]r(\lambda)=(-1)^n,
\qquad
\deg s(\lambda)=n-p.
\]
Let
\[
U=\{p+1,\dots,n\}.
\]

Recall  Case 2 of Lemma~\ref{lemliu1},   

\begin{equation}\label{eq:sec6-factor}
D(x,\lambda)
=
\prod_{\mu\in A_l}
\bigl(r_1(\lambda)+s_1(\lambda)\mu x^{a/l}+t_1(\lambda)\mu^{-1}x^{-a/l}\bigr),
\end{equation}
where
\[
r_1(\lambda)+s_1(\lambda)x^{a/l}+t_1(\lambda)x^{-a/l}
\]
is irreducible and
\[
s_1(\lambda)=\xi\, t_1(\lambda)
\qquad\text{for some }\xi\in A_l.
\]

Taking the coefficient of $x^a$ in \eqref{eq:sec6-factor}, we obtain
\begin{equation}\label{eq:sec6-sfactor}
s(\lambda)=[x^a]D(x,\lambda)
=
\Bigl(\prod_{\mu\in A_l}\mu\Bigr)s_1(\lambda)^l
=
(-1)^{l-1}s_1(\lambda)^l.
\end{equation}

Therefore
\begin{equation}\label{eq:sec6-degs}
l\,\deg s_1(\lambda)=l\,\deg t_1(\lambda)=\deg s(\lambda)=n-p,
\end{equation}
and by \eqref{eq:sec6-factor}, 
\begin{equation}\label{eq:sec6-degr}
n=l\,\deg r_1(\lambda).
\end{equation}

We now split the proof into two cases.

\smallskip
\noindent\textbf{Subcase 1: $\deg s(\lambda)\in \{1,n-1\}$.}

By \eqref{eq:sec6-degs} and \eqref{eq:sec6-degr}, the integer $l$ divides both $n$ and $\deg s(\lambda)$. Since
\[
\gcd(n,1)=\gcd(n,n-1)=1,
\]
this forces $l=1$, contradicting $l\ge 2$.

\smallskip
\noindent\textbf{Subcase 2: $\deg s(\lambda)=0$.}

In this case $p=n$  and 
$s(\lambda)$
is a nonzero constant depending only on the fixed edge weights. By \eqref{eq:sec6-sfactor}, both $s_1(\lambda)$ and $t_1(\lambda)$ are therefore nonzero constants depending only on the fixed edge weights.

Set
\[
m=\deg r_1(\lambda)=\frac{n}{l},
\qquad
r_1(\lambda)=\sum_{j=0}^{m} c_j(v)\lambda^{m-j}.
\]
For a set of variables $u$, let $\CC_k[u]$ denote the polynomials in $u$ of total degree at most $k$.

We first claim that for every $0\le k\le m$,
\begin{equation}\label{eq:sec6-coef-r1}
[\lambda^{n-k}]D(x,\lambda)=[\lambda^{n-k}]\,r_1(\lambda)^l.
\end{equation}
Indeed, in the factorization \eqref{eq:sec6-factor}, any term involving at least one of the $x^{\pm a/l}$-terms has $\lambda$-degree at most
\[
(l-1)m=n-m.
\]
Hence such terms do not contribute to $[\lambda^{n-k}]$ when $k<m$. When $k=m$, the only possible contributions of this type come from choosing one $x^{\pm a/l}$-term and the leading $\lambda^m$-term from each of the other $l-1$ factors. However, these contributions cancel because
\[
\sum_{\mu\in A_l}\mu=\sum_{\mu\in A_l}\mu^{-1}=0
\qquad (l\ge 2).
\]
This proves \eqref{eq:sec6-coef-r1}.

Next, for each $0\le k\le m$, the coefficient $[\lambda^{n-k}]D(x,\lambda)$ has the form
\begin{equation}\label{eq:sec6-left}
[\lambda^{n-k}]D(x,\lambda)
=
(-1)^{n-k}e_k(v_1,\dots,v_n)+h_k(v),
\end{equation}
where $e_k(v_1,\dots,v_n)$ is the $k$th elementary symmetric polynomial and
\[
h_k(v)\in \CC_{k-1}[v_1,\dots,v_n].
\]

From \eqref{eq:sec6-coef-r1} with $k=0$, we get
\[
c_0(v)^l=[\lambda^n]D(x,\lambda)=(-1)^n,
\]
so $c_0(v)$ is a nonzero constant.

We now prove inductively that for all $k=0,1,\dots,m,$
\[
c_k(v)\in \CC_k[v_1,\dots,v_n].
\]
Assume that this has already been proved for all indices $\leq k-1$, where $1\le k\le m$. Expanding $r_1(\lambda)^l$ gives
\begin{equation}\label{eq:sec6-right}
[\lambda^{n-k}]\,r_1(\lambda)^l
=
\sum_{\substack{j_1,\dots,j_l\in\{0,\dots,k\}\\ j_1+\cdots+j_l=k}}
c_{j_1}(v)\cdots c_{j_l}(v).
\end{equation}
Among these  terms appearing on the right side of \eqref{eq:sec6-right}, the only ones involving $c_k(v)$ are   one index   equal to $k$ and all remaining indices equal to $0$. Therefore
\begin{equation}\label{eq:sec6-right2}
[\lambda^{n-k}]\,r_1(\lambda)^l
=
l\,c_0(v)^{\,l-1}c_k(v)+Q_k(v),
\end{equation}
where
\begin{equation}\label{eq:sec6-Qk}
Q_k(v)
=
\sum_{\substack{j_1,\dots,j_l\in\{0,\dots,k-1\}\\ j_1+\cdots+j_l=k}}
c_{j_1}(v)\cdots c_{j_l}(v).
\end{equation}
By the induction hypothesis, every monomial appearing in $Q_k(v)$ has degree at most
\[
j_1+\cdots+j_l=k
\]
in the variables $v_1,\dots,v_n$. Hence
\[
Q_k(v)\in \CC_k[v_1,\dots,v_n].
\]
Combining \eqref{eq:sec6-coef-r1}, \eqref{eq:sec6-left}, and \eqref{eq:sec6-right2}, we find
\[
(-1)^{n-k}e_k(v_1,\dots,v_n)+h_k(v)
=
l\,c_0(v)^{\,l-1}c_k(v)+Q_k(v).
\]
Since the left-hand side belongs to $\CC_k[v_1,\dots,v_n]$ and $c_0(v)$ is a nonzero constant, it follows that
\[
c_k(v)\in \CC_k[v_1,\dots,v_n].
\]
This proves the claim.

If $\deg_v c_m(v)\le m-1$, then the constant term  of the right-hand side of \eqref{eq:sec6-left} has $v$-degree at most
\begin{equation}\label{gapr175}
l(m-1).
\end{equation}

 On the other hand, the  highest degree term of $v$ in $D(x,\lambda)$ is 
\begin{equation}\label{gapr174}
    \prod_{j=1}^n v_j,
\end{equation}
and hence has $v$-degree exactly $n=lm$. By \eqref{gapr175} and \eqref{gapr174}, \eqref{eq:sec6-left}  cannot hold generically.

If $\deg_v c_m(v)=m$, then on both sides of \eqref{eq:sec6-left} the highest degree in $v$ is $n$.

Let $f(v)$ be the homogeneous part of degree $m$ of $c_m(v)$.  Hence the highest homogeneous term in $v$ on the right-hand side is
\begin{equation}\label{gapr178}
    f(v)^l.
\end{equation}
By \eqref{gapr174} and \eqref{gapr178}, \eqref{eq:sec6-left} cannot hold generically.

Thus Subcase~2 cannot occur generically.

\smallskip
\noindent\textbf{Subcase 3: $2\le \deg s(\lambda)\le n-2$.}

Let
\[
n'=\deg s(\lambda)=n-p,
 \qquad
m'=\deg s_1(\lambda)=\frac{n'}{l},
\qquad
v'=(v_{p+1},\dots,v_n),
\]
and write
\begin{equation}\label{gapr171}
    s_1(\lambda)=\sum_{j=0}^{m'} b_j \lambda^{m'-j}.
\end{equation}
By \eqref{eq:sec6-sfactor}, \eqref{gapr1910} and Part B of Prop. \ref{prop},

\[
(-1)^{l-1}s_1(\lambda)^l=s(\lambda)=\kappa\,D_U(\lambda),
\]
The right-hand side depends only on the variables $v'$. Exactly the same coefficient-comparison argument as in Subcase~2 (in particular see \eqref{eq:sec6-coef-r1}), now applied to the identity
\begin{equation}\label{gapr1911}
(-1)^{l-1}s_1(\lambda)^l
=
\kappa\,D_U(\lambda).
\end{equation}

shows that generically with respect to $v'$, \eqref{gapr1911} can not happen.

Thus Subcase~3 also cannot occur generically.

We have therefore ruled out  Case 2 of Lemma~\ref{lemliu1} for a generic choice of potential.
\end{proof}

\section{Proof of main theorem and examples}
 \subsection{Proof of main theorem}
\begin{Theorem}\label{thmonecase3}
    Assume that  $\Gamma$ has a fundamental domain $W$
which is minimally connected and there exists exactly one edge between $W$ and $\Gamma\backslash W$.

After fixing non-zero edge weights,  
for a generic potential, the dispersion polynomial
$D(x,\lambda)$  of $\Gamma$ can not be factorized in  Case 3 appearing in Lemma \ref{lemliu1}. 
\end{Theorem}
\begin{proof}
    Assume $D(x,\lambda)$   factor as in  Case 3 of Lemma~\ref{lemliu1}. Thus there exist an integer $l\ge 1$ with $l\mid a$ and polynomials
\[
r_2(\lambda),\ s_2(\lambda)\in \CC[\lambda]
\]
such that
\begin{equation}\label{eq:main-proof-case3}
D(x,\lambda)
=
\prod_{\mu\in A_l}
\bigl(r_2(\lambda)+s_2(\lambda)\mu x^{a/l}\bigr)
\bigl(r_2(\lambda)+s_2(\lambda)\mu x^{-a/l}\bigr).
\end{equation}

Now specialize $x=1$. Then
\[
D(1,\lambda)=\prod_{\mu\in A_l}
\bigl(r_2(\lambda)+s_2(\lambda)\mu  \bigr)^2.
\]
Therefore $D(1,\lambda)$ has repeated roots. This  can not happen generically by Lemma \ref{lemsimple}.
\end{proof}
\begin{proof}[\bf Proof of Theorem~\ref{Thm:main2}]
By Lemma~\ref{lemliu1}, if $D(x,\lambda)$ is reducible, then one of the three cases in Lemma~\ref{lemliu1} occurs. Theorem~\ref{thmonecase1} rules out Case 1 generically, Theorem~\ref{thmonecase2} rules out Case  2 generically, and Theorem~\ref{thmonecase3} rules out  Case 3 generically. Therefore, for a generic choice of potential, $D(x,\lambda)$ is irreducible.
\end{proof}
\subsection{Examples}
In this subsection,  we are going to construct examples showing that  Theorem \ref{Thm:main}  can not be relaxed to fixed nonzero edge weights and generic potentials.
\begin{Example}

Let $\Gamma$ be the $\mathbb Z$-periodic graph with vertices

\[
u_m,\ w_m,\qquad m\in\mathbb Z,
\]

and edges: for all $m\in\mathbb Z$, 

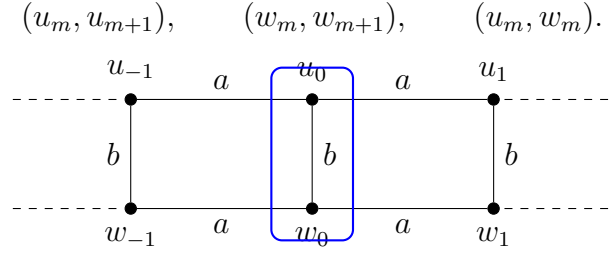
\[
(u_m,u_{m+1}),\qquad (w_m,w_{m+1}),\qquad (u_m,w_m).
\]

\begin{center}
\begin{tikzpicture}[scale=1.2,
  vertex/.style={circle,draw,fill=black,inner sep=1.5pt}]

\node[vertex] (ul) at (-2,1.2) {};
\node[above=3pt] at (ul) {$u_{-1}$};

\node[vertex] (u0) at (0,1.2) {};
\node[above=3pt] at (u0) {$u_0$};

\node[vertex] (ur) at (2,1.2) {};
\node[above=3pt] at (ur) {$u_1$};

\node[vertex] (wl) at (-2,0) {};
\node[below=3pt] at (wl) {$w_{-1}$};

\node[vertex] (w0) at (0,0) {};
\node[below=3pt] at (w0) {$w_0$};

\node[vertex] (wr) at (2,0) {};
\node[below=3pt] at (wr) {$w_1$};

\draw (ul)--(u0) node[midway,above] {$a$};
\draw (u0)--(ur) node[midway,above] {$a$};

\draw (wl)--(w0) node[midway,below] {$a$};
\draw (w0)--(wr) node[midway,below] {$a$};

\draw (ul)--(wl) node[midway,left] {$b$};
\draw (u0)--(w0) node[midway,right] {$b$};
\draw (ur)--(wr) node[midway,right] {$b$};

\draw[dashed] (-3.3,1.2)--(-2.12,1.2);
\draw[dashed] (-3.3,0)--(-2.12,0);

\draw[dashed] (2.12,1.2)--(3.3,1.2);
\draw[dashed] (2.12,0)--(3.3,0);

\draw[blue,thick,rounded corners] (-0.45,-0.35) rectangle (0.45,1.55);
\node[blue] at (0.3,-1.55) {$W=\{u_0,w_0\}$};

\end{tikzpicture}
\end{center}

Give the   horizontal edge a
nonzero weight $a\in\mathbb C^*$, and give the vertical edge  a nonzero
weight $b\in\mathbb C^*$.

Choose the fundamental domain $W=\{u_0,w_0\}$ and potentials

\[
V(u_m)=v_1,\qquad V(w_m)=v_2.
\]

Clearly, the Floquet matrix is

\[
L(x)=
\begin{pmatrix}
v_1+a(x+x^{-1}) & b\\
b & v_2+a(x+x^{-1})
\end{pmatrix},
\]

and

\[
D(x,\lambda)
=
\det(L(x)-\lambda I)
=
\bigl(v_1+a(x+x^{-1})-\lambda\bigr)
\bigl(v_2+a(x+x^{-1})-\lambda\bigr)-b^2.
\]

Let

\[
\rho_\pm
=
\frac{v_1+v_2\pm\sqrt{(v_1-v_2)^2+4b^2}}{2}.
\]

We conclude that

\[
D(x,\lambda)
=
\bigl(a(x+x^{-1})+\rho_+-\lambda\bigr)
\bigl(a(x+x^{-1})+\rho_- -\lambda\bigr).
\]

Fix any nonzero $a$ and $b$.
The graph is connected (implying the quotient graph is also connected) and   the  graph is nontrivial.

For every potential, $D(x,\lambda)$  is reducible (has two factors).
For generic potentials, $\rho_+\neq \rho_-$, so $D(x,\lambda)$ has two nonconstant distinct  factors in $\mathbb C[x^{\pm1},\lambda]$.

\end{Example}

\section*{Acknowledgements}
W. Liu was supported in part by NSF grants DMS-2246031 and DMS-2052572,  and by a Visiting Miller Professorship from the Miller Institute
for Basic Research in Science, University of California, Berkeley.

\section*{Statements and Declarations}
{\bf Conflict of Interest} 
The authors declare no conflicts of interest.

\vspace{0.2in}
{\bf Data Availability}
Data sharing is not applicable to this article as no new data were created or analyzed in this study.

\bibliographystyle{alpha} 
\bibliography{references}

\end{document}